\documentclass[12pt,twoside]{amsart}
\usepackage[dvips]{graphics}
\usepackage[all]{xy}
\usepackage{a4,amsmath,amssymb,amsfonts,amscd,mathrsfs}
\reversemarginpar
\newcounter{Abschnitt}[section]

\newcommand{\bib}[6]{{\bibitem{#2} #3: {\emph{#4},} #5#6.}}

\newtheorem{theorem}[subsection]{Theorem}
\newtheorem{lemma}[subsection]{Lemma}
\newtheorem{corollary}[subsection]{Corollary}
\newtheorem{proposition}[subsection]{Proposition}

\newcommand{\CH}{{\rm CH}}

\newcommand{\Coh}{\mathrm{Coh}}

\newcommand{\End}{{\rm End}}
\newcommand{\Ext}{{\rm Ext}}

\newcommand{\Hom}{{\rm Hom}}

\newcommand{\Pic}{{\rm Pic}}

\newcommand{\SL}{{\rm SL}}

\newcommand{\Spec}{{\rm Spec}}

\newcommand{\U}{{\rm U}}

\newcommand{\ch}{{\rm ch}}

\newcommand{\eps}{\varepsilon}

\newcommand{\gr}{{\rm gr}}
\newcommand{\id}{{\rm id}}

\renewcommand{\phi}{\varphi}
\newcommand{\pr}{{\rm pr}}

\newcommand{\pt}{{\rm pt}}
\newcommand{\rk}{{\rm rk}}

\newcommand{\Ecal}{{\mathcal E}}

\newcommand{\Ocal}{{\mathcal O}}

\newcommand{\ndop}{{\mathbb N}}

\newcommand{\pdopi}{{{\mathbb P}^1}}
\newcommand{\pdop}{{\mathbb P}}
\newcommand{\qdop}{{\mathbb Q}}

\newcommand{\ddual}{^{\lor \lor}}

\newcommand{\inv}{^{-1}}
\newcommand{\rar}{\longrightarrow}

\newcommand{\rarpa}[1]{\stackrel{#1}{\longrightarrow}}

\newcommand{\larpa}[1]{\stackrel{#1}{\leftarrow}}

\author{Georg Hein}
\email{georg.hein@uni-due.de}
\author{Thang Quyet Truong}
\email{algebraic.geometry.2011@gmail.com}
\address{Fakult\"at f\"ur Mathematik, Universit\"at Duisburg-Essen,
45117 Essen, Germany}
\title{Theta functions for Holomorphic triples}
\date{February 8, 2017}
\begin{document}

\maketitle

\begin{abstract}
  We introduce an generalization of the theta divisor to the theory of
  holomorphic triples on a smooth projective curve $X$.
  We show that a given triple $T=(E_1 \to E_0)$
  is $\alpha$-semistable iff there exists an orthogonal tripe $S=(F_1 \to
  F_0)$ with given numerical invariants.
  This yields globally generated theta line bundles on the moduli space
  of semistable triples.
\end{abstract}

\section{Introduction}
We fix a smooth projective curve $X$ of genus $g$ over an algebraically closed
ground field $k$.
When investigating the coarse moduli space $\U_X(r,d)$ of
S-equivalence classes of semistable bundles on $X$, an ample Cartier divisor
which allows a geometric interpretation facilitates its study.
Drezet and Narasimhan defined in \cite{DN} with the generalized theta
line bundle $\Ocal_{\U_X(r,d)}(\Theta)$ such a line bundle.
The generalized theta divisor has nice sections $\theta_F$ associated to
vector bundles $F$ of slope $\mu(F)=(g-1)-\frac{d}{r}$ with vanishing divisor
$\Theta_F$.
The $k$ points of this divisor are
\[ \Theta_F(k) = \left\{ [E ]\in \U_X(r,d)(k) \, | \, \mbox{ such that }
H^*(X, E \otimes F) \ne 0 \, \right\} \, .\]
The result of Faltings' about the
existence of  {\em orthogonal objects} is basic for the theory of the generalized theta
divisor:

\begin{theorem}\label{Fal-thm}(Theorem \cite[Theorem 1.2]{Fal})
  For a vector bundle $E$ on $X$ we have the following equivalence:
  \[ E \mbox{ is semistable} \iff  H^*(X, E \otimes F) =0 \mbox{ for a vector
  bundle } F \ne 0 \, .\]
\end{theorem}

This result is a key step in showing that the generalized theta line bundle
is ample.
Indeed, it shows that a certain power of the generalized theta line bundle is
globally generated.
Furthermore, as shown by Popa, the rank and the determinant of $F$ can be fixed a priori:

\begin{theorem}\label{Pop-thm}(Theorem \cite[Theorem 5.3]{Popa})
  For a vector bundle $E$ of rank $r$ and degree $d$ on $X$ we have the following equivalence
  for any fixed line bundle $L$ of degree $r^2(g-1)-r \cdot d$:
  \[ E \mbox{ is semistable} \iff  H^*(X, E \otimes F) =0 \mbox{ for a } F
\mbox{ with } \rk(F)=r^2 , \, \det(F) \cong L \, . \]
\end{theorem}

The result of Popa gives a concrete bound for the generatedness of the
generalized theta line bundle.

On the other side there is a coarse moduli space of semistable triples
introduced by Bradlow and Garc\'ia-Prada in \cite{BGP}, further studied
together with Gothen in \cite{BGPG}. This space parameterizes pairs $(E_1
\rarpa \phi E_0)$ of vector bundles $E_i$ on $X$ together with a homomorphism
$\phi$ between them. The construction of triples was extended to the case of
arbitrary characteristic by \'Alvarez-C\'onsul in \cite{AC}. 
The aim of this article is to introduce a generalized theta
divisor for holomorphic triples which also possesses
a nice geometric  description
of its vanishing divisor, as well as versions of the above theorems
\ref{Fal-thm} and \ref{Pop-thm}.

Indeed, for a triple $T=(E_1 \to E_0)$, we show in Theorem \ref{thmA}, that
the $\alpha$-semistability of $T$ is equivalent to the existence of an
orthogonal triple $S=(F_1 \to F_0)$ satisfying certain numerical conditions
which involve the parameter $\alpha$.
Orthogonality means, that the morphism
\[\xymatrix{(E_1 \otimes F_0) \oplus (E_0 \otimes F_1) \ar[rrrr]^-{
        -\phi\otimes \id_{F_0} + \id_{E_0} \otimes \psi } &&&& E_0
  \otimes F_0 }\]
  is surjective and its kernel $K$ satisfies $H^*(X,K)=0$. 

This allows to define the set of all triples $T$ such that $T$ is not
orthogonal to $S$.
As we see in 
Section \ref{S6} this set is a Cartier divisor $\Theta_S$
corresponding to  the generalized
theta line bundle on the moduli space of holomorphic triples.
We see in Proposition \ref{thmB2} that we can bound the ranks in the
triple $S$ independent of $T$. This result is the equivalent to Popa's base
point free theorem \ref{Pop-thm}.

{\bf The plan of the article} is as follows: 
We start in Section \ref{S2} with the equivalence of triples on $X$ and
certain short exact sequences on $X \times \pdopi$. This is similar to the
approach in \cite{BGP}. Next in Section \ref{S3}, we show the Bogomolov
inequality for $X \times
\pdopi$. Since, we show it in any characteristic this result is new in
positive characteristic.
In Section \ref{S4} we explain how this induces effective restriction theorems
on $X \times \pdopi$. Indeed, we almost follow the  book \cite{HL} of Huybrechts and
Lehn in doing so. Section \ref{S5} contains the main result of this article
the equivalence of semistability and orthogonality for triples. Having done
this, we define in Section \ref{S6} the generalized theta divisor $\Theta_R$
for a triple. Here we follow the exposition in the article \cite{DN} of Drezet and
Narasimhan. In the last Section we present a simple example of a theta divisor
$\Theta_R$ on a specific moduli space of triples such that $\Theta_R$ is ample and the
linear system $|2 \cdot \Theta_R|$ is globally generated.

{\bf Notation.}
We fix a smooth projective curve $X$  of genus $g$ over an algebraically closed field $k$.
For a triple $T = (E_1 \rarpa{\phi} E_0)$ we denote the ranks by $r_1$ and
$r_0$, and analogously the degrees by $d_1$ and $d_0$, respectively. 
The $\pdopi$ is the space $\pdop(V)$ for the two dimensional $k$ vector space $V$.
We denote  the projections from $X \times \pdopi$ as follows:
$\xymatrix{X & X \times \pdopi \ar[l]_-p \ar[r]^-{q} & \pdopi}$.
We use $F_p$ and $F_q$ for the classes of the fibers of points with respect to $p$ and $q$ 
in the numerical Chow group $\CH^1(X \times \pdopi)$,
  $\pt$ for the class of a $k$-point in the numerical Chow group
$\CH^2(X \times \pdopi)$.
For a fixed positive rational number $\alpha$ we denote by 
$H_\alpha$  the rational polarization $H_\alpha=F_q+\alpha \cdot F_p$ on $X \times
\pdopi$.

\subsection*{Acknowledgment}
This work has been supported by the SFB/TR 45
``Periods, moduli spaces and arithmetic of algebraic varieties''.

\section{From Triples to sheaves on $X \times \pdopi$}\label{S2}

\subsection{From $T$ to $E_T$}\label{22}
We fix a triple $T=(E_1 \rarpa \phi E_0)$ on the curve $X$.
We have on $\pdopi$ the Euler sequence
\[ 0 \rar \Ocal_\pdopi(-2) \to V \otimes \Ocal_\pdopi(-1) \to \Ocal_\pdopi
\rar 0 \, .\]
So we have on $X \times \pdopi$ two morphisms to $p^*E_0$ as indicated
in the next diagram. The pull back of both morphisms we name $E_T$.
\[ \xymatrix{ E_T \ar@{-->}[r]^\beta \ar@{-->}[d] & p^*E_1 \ar[d] \\
p^*E_0 \otimes q^*(V
\otimes \Ocal_\pdopi(-1)) \ar[r] & p^*E_0 }
\]
We obtain a short exact sequence on $X \times \pdopi$
\[ 0 \rar E_0 \boxtimes \Ocal_\pdopi(-2) \rar E_T \rar p^*E_1 \rar 0 .\]
So we compute the (numerical) Chern character 
\[ \ch(E_T)= (r_0+r_1)+((d_0+d_1)F_p-2r_0F_q) - 2d_0 \pt . \]
To obtain the numerical invariants of $E_T$, as
\[ \rk(E_T)=r_0+r_1, \quad
c_1(E_T)=(d_0+d_1)F_p-2r_0 F_q, \quad \mbox{and }
c_2(E_T)=2d_0-2r_0(d_0+d_1).
\]
The degree with respect to the polarization $H_\alpha$ is
\[\deg(E_T) = \deg_{H_\alpha}(E_T) = c_1(E_T).H_\alpha = d_0+d_1-2\alpha r_0.\]

Thus, we have motivated the definition of the rank, $\alpha$-degree, and
$\alpha$-slope of the triple $T=(E_1 \rarpa{\phi} E_0)$
\[ \rk(T)=r_0+r_1, \qquad \deg_\alpha(T)=d_0+d_1-2\alpha r_0, \qquad
\mu_\alpha(T) =\frac{\deg_\alpha(T)}{\rk(T) } . \]

Once this is done we proceed to the definition of (semi)stability.

\begin{definition}
The triple $T=(E_1 \rarpa{\phi} E_0)$ is $\alpha$-(semi)stable $\iff$ the
sheaf $E_T$ is (semi)stable with respect to the polarization $H_\alpha$ on
$X \times \pdopi$.
\end{definition}

\begin{lemma}\label{lemmaET-T}
For a triple $T=(E_1 \rarpa{\phi} E_0)$ the following is equivalent:
\begin{enumerate}
\item The triple $T$ is $\alpha$-semistable.
\item For all sub triples $\xymatrix{ T' \ar@{^(->}[r] & T}$ we have
  the inequality $\mu_\alpha(T') \leq \mu_\alpha(T)$.
\item For all quotient triples $\xymatrix{ T \ar@{->>}[r] & T''}$ we have
  the inequality $\mu_\alpha(T) \leq \mu_\alpha(T'')$.
\end{enumerate}
\end{lemma}

\begin{proof}
(1) $\implies$ (2).
This direction is easy. If $\xymatrix{ T' \ar@{^(->}[r] & T}$ destabilizes
$T$, then $E_{T'} \subset E_T$ destabilizes $E_T$.
The equivalence of (2) and (3) is standard.\\
(2) $\implies$ (1).
To see this we consider $X \times \pdopi$ with the group action of $\SL_2$ 
where the action on $X$ is trivial, and the action on $\pdopi$ comes from an
identifications $\SL_2=\SL(V)$ and $\pdopi=\pdop(V)$.
Since the Euler sequence on $\pdopi$ is a sequence of $\SL_2$-bundles, we
obtain that $E_T$ is also a $\SL_2$-bundle, and the short exact sequence
\[ 0 \to p^*E_0 \otimes q^*\Ocal_\pdopi(-2) \to E_T \to p^*E_1 \to 0 \]
is a sequence of $\SL_2$-bundles. Assume that $E_T$ is not semistable, the
there exists a unique destabilizing subsheaf $F \subset E_T$. The
uniqueness  implies that $F$ is $\SL_2$-invariant, or $F$ is a
$\SL_2$-sub bundle of $E_T$. We conclude that we have the following diagram
where all sheaves are $\SL_2$-bundles, and all morphisms $\SL_2$-morphisms:
\[ \xymatrix{0 \ar[r] & F' \ar[d] \ar[r] & F \ar[d] \ar[r] & F'' \ar[d] \ar[r] & 0\\
0 \ar[r] & p^*E_0 \otimes q^*\Ocal_\pdopi(-2) \ar[r] & E_T \ar[r] &  p^*E_1
\ar[r] &  0}
\]
We deduce that $F'' = p^*F_1$ and $F'=p^*F_0 \otimes \Ocal_\pdopi(-2)$.
The long exact sequence for $p_*$ yields
the diagram
\[ \xymatrix{ F_1 \ar[r] \ar[d] & F_0 \ar[d]\\ E_1 \ar[r]^\phi & E_0} \]
with injective vertical morphisms. Thus, $(F_1 \to F_0)$
destabilizes $(E_1 \to E_0)$.
\end{proof}


\section{The Bogomolov inequality for $X \times \pdopi$}\label{S3}
\subsection{The Harder-Narasimhan functor $E \mapsto E_{(\mu)}$}
We recall that for any vector bundle $E$ on $X$ we have the unique Harder-Narasimhan
filtration \[ 0 = E_0 \subset E_1 \subset \ldots \subset E_k \]
which satisfies $E_i/E_{i-1}$ is semistable for all $i=1,\ldots ,k$, and
the rational numbers $\mu_i=\mu(E_i/E_{i-1})$ are strictly decreasing.
The set $\{\mu_1,\mu_2,\ldots, \mu_k \}$, is the set of slopes appearing in
the associated graded object $\gr(E) = \bigoplus_{i=1}^k E_i/E_{i-1}$.
For a rational number $\mu$ we define
\[ E_{(\mu)} := E_j \quad \mbox{ where } j= \max\{ i=1,\ldots,k \, |\, \mu_i
\geq \mu \} . \]
The advantage of the notation is, that
we obtain a functor $\Coh(X) \to \Coh(X)$ sending $E \mapsto E_{(\mu)}$ for all $\mu
\in \qdop$. This follows from the next
\begin{lemma}\label{lemma21}
Let $E \rarpa{\phi} F$ be a morphism of vector bundles, and $\mu \in \qdop$,
the we obtain a natural morphism $E_{(\mu)} \rarpa{\phi_{(\mu)}} F_{(\mu)}$
making the diagram
\[ \xymatrix{ E_{(\mu)} \ar[r]^{\phi_{(\mu)}} \ar[d] &  F_{(\mu)}\ar[d] \\
E \ar[r]^\phi & F }\]
commutative. The vertical arrows are the natural inclusions.
\end{lemma}
\begin{proof}
The natural inclusions give a composite homomorphisms
\[ \alpha: E_{(\mu)} \to E \rarpa{\phi} F \to F/F_{(\mu)} \, .\]
Since there are no homomorphism of semistable sheaves of slope $\geq \mu$ to
those of slope $< \mu$ we conclude that $\alpha $ is zero. This
induces the assertion.
\end{proof}

\begin{lemma}\label{lemma22}
Let $E$ and $F$ two vector bundles on $X$. If $\mu(E) > \mu(F)$ , then
there exists a rational number $\mu$ such that
$\frac{\rk(E_{(\mu)})}{\rk(E)} > \frac{\rk(F_{(\mu)})}{\rk(F)}$ holds.
\end{lemma}
\begin{proof}
Let  $\{\mu_1>\mu_2>\ldots> \mu_k \}$ be the union of the slopes  appearing
in the associated graded objects $\gr(E)$ and $\gr(F)$. We see that
\[ \mu(E)=\frac{1}{\rk(E)} \sum_{i=1}^k
(\rk(E_{(\mu_i)}-\rk(E_{(\mu_{i-1})})) \mu_i
=\sum_{i=1}^{k-1} \frac{\rk(E_{(\mu_i)})}{\rk(E)}(\mu_i-\mu_{i+1}) + \mu_k
 \, ,
\]
and obtain a similar formula for $\mu(F)$. From both formulas we deduce
that
\[
 0<\mu(E)-\mu(F) = \sum_{i=1}^{k-1}
 \left( \frac{\rk(E_{(\mu_i)})}{\rk(E)}-\frac{\rk(F_{(\mu_i)})}{\rk(F)}
 \right)
 (\mu_i-\mu_{i+1}) .
\]
Since the differences $(\mu_i - \mu_{i+1})$ are all positive, the statement
follows.
\end{proof}

\begin{theorem}\label{bog-thm}
{\bf (Bogomolov inequality for curves times $\pdop^1$ in any
characteristic)}
Let $X$ be a smooth projective curve over some algebraically closed
field $k$. Let $E$ be a vector bundle on $X \times \pdop^1$. We consider
the discriminant $\Delta(E)$, which is the number
\[ \Delta(E) = \int_{X \times \pdopi} (\rk(E)-1)c_1^2(E)-2 \rk(E) c_2(E) \, . \]
If $E$ is semistable with respect to one polarization $H$ on $X \times
\pdop^1$, then we must have
\[ \Delta(E) \leq 0 \, . \]
Indeed, if $\Delta(E)>0$, then there exists a subsheaf $E'
\subset E$ such that $E'$ contradicts the semistability, that means
$\mu_H(E') > \mu_H(E)$, with respect to any polarization $H$ on $X$.
\end{theorem}
\begin{proof}
We assume that $E$ is
a vector bundle satisfying $\Delta(E) > 0$ and construct a subsheaf $E'$
of $E$ contradicting the semistability for all polarizations $H$.
To make the proof more accessible, we divide it into steps.\\
{\bf Step 1: Notation.}
We consider the morphisms
$X \larpa{p} X \times \pdop^1 \rarpa{q} \pdop^1$.
First we
remark, that the discriminant is unchanged when we twist $E$ with any
line bundle $L$. Since semistability also remains untouched when passing
from $E$ to $E \otimes L$, we may assume that $R^1p_*(E \otimes
q^*\Ocal_{\pdop^1}(-1)) =0$. The resolution of the diagonal
\[ 0 \to \Ocal_{\pdop^1}(-1) \boxtimes \Ocal_{\pdop^1}(-1) \to
\Ocal_{\pdop^1 \times \pdop^1} \to \Ocal_{\Delta_{\pdop^1}} \to 0 \]
when pulled back to $X$ yields a resolution
\[ 0 \to p^*E_1 \otimes q^*\Ocal_{\pdop^1}(-1) \to p^* E_0 \to E \to 0
\, .\]
Using this short exact sequence, and using the numbers $r_i=\rk(E_i)$ and
$d_i=\deg(E_i)$ for $i \in \{ 0, 1\}$, we find that
\[ \Delta(E)= 2(r_0d_1-r_1d_0) = 2 r_0r_1  (\mu(E_1)-\mu(E_0)) \, . \]
Thus, we conclude from $\Delta(E)>0$ that $2r_0r_1>0$ and
$\mu(E_1) > \mu(E_0)$ hold.

{\bf Step 2: Choice of $\mu$.}
Since $\mu(E_1) > \mu(E_0)$ we conclude from Lemma \ref{lemma22} that
there exists a rational number $\mu$ such that
$\frac{\rk(E_{1(\mu)})}{\rk(E_1)} > \frac{\rk(E_{0(\mu)})}{\rk(E_0)} $
holds.
Now we choose $\mu$ in such a way that the quotient
$\frac{\rk(E_{1(\mu)})}{\rk(E_{0(\mu)})}$ becomes maximal. This implies that
$\mu(E_{1(\mu)} ) \leq \mu(E_{0(\mu)})$.
Indeed, if $\mu(E_{1(\mu)} ) > \mu(E_{0(\mu)})$, then we could apply
Lemma \ref{lemma22} again, to deduce that for a rational number $\nu$
we have $\frac{\rk(E_{1(\mu)(\nu)})}{\rk(E_{0(\mu)(\nu)})} >
\frac{\rk(E_{1(\mu)})}{\rk(E_{0(\mu)})}$. However $E_{1(\mu)(\nu)} =
E_{1(\max\{\mu,\nu\})}$ and analogously for $E_{0(\mu)( \nu)}$.
We have the two inequalities:
\begin{eqnarray}\label{eq1}
  \frac{\rk(E_{0(\mu)})}{\rk(E_{1(\mu)})} < \frac{\rk(E_{0})}{\rk(E_1)}
\qquad \mbox{ and} \qquad \mu(E_{1(\mu)} ) \leq \mu(E_{0(\mu)}) .
\end{eqnarray}

{\bf Step 3: The subsheaf $E' \subset E$.}

When restricting to a section of $q$, we see that $E_1$ is a
subsheaf of $E_0$. We deduce from Lemma \ref{lemma21} that
$E_{1(\mu)}(-1)$ is also a subsheaf of $E_{0(\mu)}$.

We obtain the following commutative diagram with exact rows
\begin{eqnarray}\label{eq2} && \xymatrix{
0 \ar[r]
& p^*E_{1\mu} \otimes q^* \Ocal_{\pdop^1}(-1)  \ar[d] \ar[r]
& p^*E_{0\mu} \ar@{^(->}[d] \ar[r]
& E'' \ar[d]^\alpha \ar[r] & 0\\
0 \ar[r]
& p^*E_{1} \otimes q^* \Ocal_{\pdop^1}(-1)  \ar[r]
& p^*E_{0} \ar[r]
& E \ar[r] & 0 \, .\\
}\end{eqnarray}
Now we define $E'$ to be the image of $\alpha$. From the long kernel
cokernel sequence we obtain
\begin{eqnarray}\label{eq3} && 0 \to \ker(\alpha) \to E'' \to E' \to 0 \quad \mbox{and} 
\quad \ker(\alpha) \subset p^*(E_1/E_{1(\mu)}) \otimes q^*
\Ocal_{\pdop^1}(-1) \, . \end{eqnarray}
When we restrict $E''$ to a fiber $F_p$ of $p$ we obtain that
$E''|_{F_p}$ is globally generated. Thus, the kernel of $\alpha$ which
has no global sections is strictly smaller. We deduce that $E'$ is not
the zero subsheaf of $E$. In particular, its rank is positive.
To show that $E'$ is destabilizing, we use that up to numerical
equivalence every ample class on $X$ is of type $aF_q+bF_p$ with $a$ and
$b$ positive integers and $F_q$, $F_p$ the fibers of $q$ and $p$.
Thus, to show that $E'$ is destabilizing with respect to any
polarization it suffices to show that it destabilizes when restricted to
$F_q$ and $F_p$.

{\bf Step 4: The subsheaf $E'$ destabilizes $E$ when restricted to $F_p$.}
We use the short exact sequences of (\ref{eq2}) and (\ref{eq3}) to compute the
ranks and degrees of $E|_{F_p}$, $E''|_{F_p}$, and $\ker(\alpha)|_{F_p}$.
\[ \begin{array}{rclcrcl}
\rk(E|_{F_p}) & =& \rk(E_0) - \rk(E_1)& \quad&
\deg(E|_{F_p})&=&\rk(E_1)\\
    \rk(E''|_{F_p}) & =& \rk(E_{0(\mu)}) - \rk(E_{1(\mu)})& \quad&
    \deg(E''|_{F_p})&=&\rk(E_{1(\mu)}) \\
    \rk(\ker(\alpha)|_{F_p}) & =& r_\alpha & \quad&
    \deg(\ker(\alpha)|_{F_p})&\leq &-r_\alpha \, .\\
\end{array}
\]
We deduce that the slope of $E'|_{F_p}$ is given by
\[ \mu(E'|_{F_p}) = 
  \frac{-\deg(\ker(\alpha)|_{F_p})+\rk(E_{1(\mu)})}{\rk(E_{0(\mu)}) -
  \rk(E_{1(\mu)})-r_\alpha}
  \geq \frac{r_\alpha+\rk(E_{1(\mu)})a}{\rk(E_{0(\mu)})) -
  \rk(E_{1(\mu)})-r_\alpha}
\]
The function on the left is strictly increasing in $r_\alpha$, which satisfies
$r_\alpha \geq 0$. Thus, we conclude that 
\[ \mu(E'|_{F_p})\geq \frac{\rk(E_{1(\mu)})}{\rk(E_{0(\mu)})) -
\rk(E_{1(\mu)})}  = \frac{1}{\frac{\rk(E_{0(\mu)})}{\rk(E_{1(\mu)})}-1}\, . \]
The last inequality together with the inequality (\ref{eq1}) gives
\[ \mu(E'|_{F_p}) > \frac{1}{\frac{\rk(E_{0})}{\rk(E_{1})}-1} =
\mu(E|_{F_p}) . \]
  Therefore, the subsheaf $E'$ is destabilizing with respect to the fibers of
  $p$.

{\bf Step 5: The subsheaf $E'$ destabilizes $E$ when restricted to $F_q$.}
We will repeatedly use the following formula for a short exact sequence of vector bundles
\[ 0 \to G' \to G \to G'' \to 0 \]
on $X$ which gives the slope of $G$ in terms of the slopes of $G'$ and $G''$:
\[ \mu (G) = \frac{\rk(G')}{\rk(G') +\rk(G'')} \mu(G') +
  \frac{\rk(G'')}{\rk(G') +\rk(G'')} \mu(G'') \, .  \]
Thus, $\mu(G)$ is a weighted average of $\mu(G')$ and $\mu(G'')$.
We conclude that any relation $(<,\leq,=,\geq,>)$ between  $\mu(G')$ and
$\mu(G)$ implies the same relation between $\mu(G)$ and $\mu(G'')$.

The short exact sequence from (\ref{eq2}) restricted to $F_q$ yields
\begin{eqnarray}\label{u1} \mu(E_0) > \mu(E|_{F_q})\end{eqnarray}
  since we have $\mu(E_1) > \mu(E_0)$.
  Since $E_{0(\mu)}$ appears in the Harder-Narasimhan filtration of $E_0$ we have
\begin{eqnarray}\label{u2}\mu(E_{0(\mu)} ) \geq \mu(E_0) \, . \end{eqnarray}
  By the definition of $E_{0(\mu)}$ we conclude that
\begin{eqnarray}\label{u3} \mu(E_{0(\mu)} ) \geq \mu \, .\end{eqnarray}
  The second inequality in (\ref{eq1}) with the short exact sequence from
  (\ref{eq2}) restricted to $F_q$ yields
\begin{eqnarray}\label{u4} \mu(E_{0(\mu)} ) \leq \mu(E''|_{F_q}) \, .\end{eqnarray}
The inclusion $\ker(\alpha)  \subset p^*(E_1/E_{1(\mu)}) \otimes q^*
\Ocal_{\pdop^1}(-1)$ from (\ref{eq3}) 
when restricted to $F_q$ yields $\ker(\alpha)|_{F_q} \subset E_1/E_{1(\mu)}$.
The graded object of $ E_1/E_{1(\mu)}$ consists of semistable sheaves of slope
strictly smaller than $\mu$.
So it follows, that $\mu(\ker(\alpha)|_{F_q}) < \mu$. This gives together with 
$\mu(E''|_{F_q}) \geq \mu$ which is a conclusion of inequalities (\ref{u3})
and (\ref{u4}) that
\begin{eqnarray}\label{u5}
\mu(E'|_{F_q}) \geq \mu(E''|_{F_q})  \, .\end{eqnarray}
The inequalities (\ref{u1}), (\ref{u2}), (\ref{u4}), and (\ref{u5}) yield 
$\mu(E'|_{F_q}) > \mu(E|_{F_q})$.
This finishes the proof.
\end{proof}

\section{Bogomolov's restriction theorem for $X \times \pdopi$}\label{S4}
The next result is a standard conclusion how Bogomolov's inequality (Theorem
\ref{bog-thm}) induces effective restriction theorems. We follow the
presentation of Section 7.3 in the book \cite{HL} of Huybrechts and Lehn.

\begin{proposition}\label{bog-res-prop}(cf.~\cite[Theorem 7.3.5]{HL})
  Let $E$ be a stable vector bundle of rank $r \geq 2$ on $X \times
  \pdopi$. Let $H$ be an
  ample divisor on $X \times \pdopi$ and $k \in \ndop$ a number such that \\
  \begin{tabular}{ll}
    (1) & the linear system $|k \cdot H|$ contains smooth curves, and\\
    (2) & we have $k  > \frac{r-1}{r}(-\Delta(E))+\frac{1}{(r-1) r \cdot H^2}$.\\
  \end{tabular}\\
  Then the restriction of $E$ to any smooth curve $C$ in the linear system
  $|k \cdot H|$ is stable.
\end{proposition}
\begin{proof}

Assume that $C \in | k \cdot H|$ is a smooth divisor  such that $E|_C$ is not
stable. We have a short exact sequence  of vector bundles on $C$
\[ 0 \to G_1 \to E|_C \to G_2 \to 0\]
with $\mu(G_2) \leq \mu(E|_C)$.
We let $E'$ be the kernel of $E \to E|_C \to G_2$ and compute
\[ \begin{array}{rcl}
    \Delta(E') &  = &  \Delta(E) +C^2 (r \cdot \rk_C(G_2)-\rk_C(G_2)^2) + 2
r\rk_C(G_2) \left( \mu(E|_C) - \mu(G_2) \right) \\
\end{array} \]

Since $(r \cdot \rk_C(G_2)-\rk_C(G_2)^2)  \geq (r-1)$ and the last
summand is not negative we deduce that
\[\begin{array}{rcl}
    \Delta(E') & \geq & \Delta(E) + (r-1) C^2 \, . \\
\end{array}
\]
From assumption (2) if follows that $k^2 H^2(r-1) + \Delta(E) > 0$.
We deduce that $E'$ has a positive discriminant and is therefore unstable by
Bogomolov's Theorem \ref{bog-thm}. We consider the Harder-Narasimhan
filtration of $E'$ with respect to the polarization $H$.
\[ 0 = E'_0 \subset E'_1 \subset E'_2 \subset \ldots \subset E'_k =E' \, . \]
The graded objects $E'_i/E'_{i-1}$ we denote by $F_i$. We have the equality
\[ \frac{\Delta(E')}{r}  - \sum_{i=1}^k \frac{\Delta(F_i)}{\rk(F_i)}
    = \frac{1}{r}
  \sum_{1 \leq i < j \leq k}
  \rk(F_i) \rk(F_j) \left( \frac{c_1(F_i)}{\rk(F_i)} -
\frac{c_1(F_j)}{\rk(F_j)} \right)^2 \, .\]
Since the $F_i$ are semistable, we have that $\Delta(F_i) \leq 0$, so we
conclude
\[ \Delta(E') \leq \sum_{1 \leq i < j \leq k}
  \rk(F_i) \rk(F_j) \left( \frac{c_1(F_i)}{\rk(F_i)} -
  \frac{c_1(F_j)}{\rk(F_j)} \right)^2 \, .\]
  Now multiplying both sides with $H^2$ and using the Hodge index Theorem
  yields
  \[ H^2 \cdot \Delta(E') \leq \sum_{1 \leq i < j \leq k}
  \rk(F_i) \rk(F_j) \left( \mu(F_i) - \mu(F_j) \right)^2 \quad \mbox{with } 
\quad \mu(F_i) = \frac{c_1(F_i).H}{\rk(F_i)}\, .\]
Now Lemma 1.4 of \cite{Langer} gives an upper bound of the right hand side of
that inequality. So we obtain
\[ H^2 \cdot \Delta(E') \leq r^2 (\mu(F_1) - \mu(E'))(\mu(E')-\mu(F_k))
\, .\]
From the short exact sequence
$0 \to E' \to E \to G_2 \to 0$ we deduce that
$F_1$ is a subsheaf of the stable sheaf $E$. Hence we have
\[ \mu(F_1)-\mu(E') \leq \frac{\rk(G_2)}{r} C.H -\frac{1}{r(r-1)} \, .\]

From the short exact sequence
$0 \to E(-C) \to E' \to G_1 \to 0$ we see that the quotient $F_k$ of $E'$
contains a quotient of the stable sheaf $E(-C)$ of rank $\rk(F_k)$. Therefore,
we have
\[
  \mu(E' )-\mu(F_k) \leq  \frac{\rk(G_1)}{r}  C.H -\frac{1}{r(r-1)}
\]
From the last three inequalities we deduce, using $\rk(G_1)+\rk(G_2)=r$ that

\[ H^2 \cdot \Delta(E') \leq (r-\rk(G_2))\rk(G_2) (C.H)^2 - \frac{r}{r-1}C.H
  +\frac{1}{(r-1)^2}
\]
Multiplying the equality for $\Delta(E')$ with $H^2$, and using
$\mu(E|_C)-\mu(G_2) \geq 0$ we get the inequality:
\[ H^2 \cdot \Delta(E')  \geq H^2 \cdot \Delta(E) +H^2\cdot C^2
(r-\rk(G_2))\rk(G_2) \, . \]
Putting both together gives
\[ H^2 \cdot (- \Delta(E)) \geq \frac{r}{r-1} C.H - \frac{1}{(r-1)^2} \]
Dividing both sides by $H^2$ gives
\[ -\Delta(E) \geq  \frac{r \cdot k}{r-1} -\frac{1}{H^2(r-1)^2} \, .\]
This violates assumption (2). This contradiction shows that $E|_C$ is stable.
\end{proof}

\begin{theorem}\label{BogStab}
Let $E$ be a semistable coherent sheaf of rank $r \geq 2$ on $X \times
\pdopi$. Let $H$ be a an ample divisor on $X \times \pdopi$ and $k$ be an
integer such that\\
 \begin{tabular}{ll}
       (1) & the linear system $|k \cdot H|$ contains smooth curves, and\\
       (2) & we have $k  > \frac{r-1}{r}(-\Delta(E))+\frac{1}{(r-1) r \cdot
 H^2}$.\\
   \end{tabular}\\
If $0 = E_0 \subset E_1 \subset \ldots \subset E_k=E$ is the
Jordan-H\"older filtration of $E$, then  for a  general smooth curves $C \in |k \cdot
H|$  the sheaf $E_i/E_{i-1}$ restricted to $C$ is a vector bundle for all
$i=1,\ldots k$, and the Jordan-H\"older filtration of $E|_C$ is given
$0 = E_0 \subset E_1|_C \subset \ldots \subset E_k|_C=E|_C$. In particular:
$E|_C$ is semistable for a general curve in $| k \cdot H|$.
\end{theorem}
\begin{proof}
  We start with the remark, that for a coherent torsion free sheaf $E$, the
  double dual $E\ddual$ satisfies $\Delta(E\ddual) \geq \Delta(E)$, and the
  injection $E \to E\ddual$ has cokernel $T$ of dimension zero.
  Let $C$ be a smooth curve in $| k \cdot H |$. The action of $\SL_2$ on the
  homogeneous space $\pdopi$ allows to move $C$ on $X \times \pdopi$ to a curve
  which does not contain any of the points of the support of $T$. So we may
  assume that all coherent sheaves which appear in the proof are vector
  bundles.

  Denoting the graded objects $E_i/E_{i-1}$ by $F_i$ we obtain as in the proof
  of Proposition \ref{bog-res-prop}
  that 
  \[ \sum_{i=1}^k \frac{-\Delta(F_i)}{\rk(F_i)} =
    \frac{-\Delta(E)}{\rk(E)} +\frac{1}{\rk(E)} \sum_{1 \leq i < j \leq k} 
  \left( \frac{c_1(F_i)}{\rk F_i}-\frac{c_1(F_j)}{\rk F_k} \right)^2 \, .\]
 Since we have a Jordan-H\"older filtration we have that the slope of the
 $F_i$ with respect to $H$ are all the same, or $\left( \frac{c_1(F_i)}{\rk
 F_i}-\frac{c_1(F_j)}{\rk F_k} \right).H=0$. From the Hodge index theorem we
 get that  $\left( \frac{c_1(F_i)}{\rk
      F_i}-\frac{c_1(F_j)}{\rk F_k} \right)^2 \leq 0$.
      We get the inequality
      \[ \sum_{i=1}^k \frac{-\Delta(F_i)}{\rk(F_i)} \leq 
      \frac{-\Delta(E)}{\rk(E)}  \, .\]

      Multiplying by $(r-1)$ this yields 
      \[ \sum_{i=1}^k \frac{r-1}{\rk(F_i)} (-\Delta(F_i)) \leq 
      \frac{(r-1)}{r} (-\Delta(E)) \, .\]
      The summands on the left hand side are all non negative by Proposition
      \ref{bog-res-prop}. The number $k$ is greater than the right hand side
      by assumption, and so we have that $k$ is greater than 
      each summand, that is
      \[ k > \frac{r-1}{\rk(F_i)} (-\Delta(F_i)) >  \frac{\rk(F_i)-1}{\rk(F_i)}
    (-\Delta(F_i)) + \frac{1}{\rk(F_i)(\rk(F_i)-1)H^2} \, .\]
    By Proposition \ref{bog-res-prop} all the $F_i$ are stable when
    restricted to $C \in | k \cdot H|$.
\end{proof}

\section{Existence of orthogonal triples}\label{S5}
Let $\alpha  \in \qdop$ be a rational number. We say that a triple $F_1
\rarpa{\psi} F_0$ is of $\alpha$-orthogonal-type (in short is of type
$\alpha^\bot$), if the following three conditions hold
\begin{itemize}
\item the morphism $\psi$ is surjective,
\item for the ranks we have the relation $\rk(F_1)=2 \cdot \rk(F_0)$, and
\item  the slopes differ by $\alpha$, i.e. $\mu(F_0)-\mu(F_1)=\alpha$.
\end{itemize}

We say that two triples $(E_1 \rarpa \phi E_0)$ and $(F_1
\rarpa{\psi} F_0)$ are orthogonal when the morphism
\[ \xymatrix{(E_1 \otimes F_0) \oplus (E_0 \otimes F_1) \ar[rrrr]^-{\pi=
  -\phi\otimes \id_{F_0} + \id_{E_0} \otimes \psi } &&&& E_0
\otimes F_0 }\]
is surjective, and we have $H^*(X, \ker(\pi))=0$. 

\begin{theorem}\label{thmA}
For any $\alpha \in \qdop_+$ we have the following two equivalent conditions for any
triple
$T=(E_1 \rarpa \phi E_0)$ on our curve $X$:
\begin{itemize}
\item[(1)]  $T$  is $\alpha$-semistable.
\item[(2)] there exists an orthogonal
triple $S=(F_1 \rarpa \psi F_0)$ of type $\alpha^\bot$.
\end{itemize}
\end{theorem}

The proof of this theorem will occupy this section.
However the implication (2) $\implies$ (1) in Theorem \ref{thmA} is not so
hard to see. We show it in Corollary \ref{thmA1}. The harder inclusion
(1)$\implies$(2) is Proposition \ref{thmA2} which relies on results of all
previous sections.  

\begin{lemma}\label{32}
Let $S=(F_1 \rarpa \psi F_0)$ be a triple of type $\alpha^\bot$. 
Then for any triple $T=(E_1 \rarpa \phi E_0)$ the morphism
\[ \xymatrix{(E_1 \otimes F_0) \oplus (E_0 \otimes F_1) \ar[rrrr]^-{\pi=
  -\phi\otimes \id_{F_0} + \id_{E_0} \otimes \psi } &&&& E_0
  \otimes F_0 }\]
is surjective, and we have $\mu( \ker(\pi))=\mu_\alpha(T)+\mu(F_0)$. 
\end{lemma}
\begin{proof}
The surjectivity of $F_1 \to F_0$ implies that $E_0 \otimes F_1 \to E_0
\otimes F_0$ is surjective. Hence $\pi$ is surjective.
Now the computation of the numerical invariants of $\ker(\pi)$ is
straightforward and yields
\[
\begin{array}{ccl}
  \rk(\ker \pi) & =&
  \rk(F_0) \cdot \left( \rk(E_1)+\rk(E_0) \right) \\
  \deg(\ker \pi) & =& 
\rk(F_0) \cdot \left(\deg(E_1)+\deg(E_0)-2\alpha \cdot \rk(E_0) \right)
+\left(\rk(E_1)+\rk(E_0) \right) \cdot \deg(F_0)\\
&=& \rk(\ker \pi)\cdot \left( \mu_\alpha(T) + \mu(F_0) \right) \, .
\end{array}
\]
This is the statement of the lemma.
\end{proof}

\begin{corollary}\label{thmA1}
Let $T=(E_1 \rarpa \phi E_0)$ be a given triple. If $T$ is orthogonal to
the triple $S=(F_1 \rarpa \psi F_0)$ of type $\alpha^\bot$, then 
$T$ is $\alpha$-semistable.
\end{corollary}
\begin{proof}
Assume the triple $S$ is of type $\alpha^\bot$ and orthogonal to $T$. For
any sub object $T'=(E_1' \rarpa {\phi'} E_0') \subset T$ we obtain the
following commutative diagram with injective vertical arrows:
\[ \xymatrix{
0 \ar[r] & \ker(\pi') \ar[r]\ar[d] &
(E_1' \otimes F_0) \oplus (E_0' \otimes F_1) \ar[d]
\ar[rrr]^-{\pi'=-\phi'\otimes \id_{F_0} + \id_{E_0'} \otimes \psi } &&&
  E_0' \otimes F_0 \ar[r]\ar[d] & 0\\
0 \ar[r] & \ker(\pi) \ar[r] &
(E_1 \otimes F_0) \oplus (E_0 \otimes F_1)
\ar[rrr]^-{\pi=-\phi\otimes \id_{F_0} + \id_{E_0} \otimes \psi } &&&
E_0 \otimes F_0 \ar[r] & 0\\
}\]
We study now the short exact sequence $0 \to \ker \pi' \to \ker \pi \to Q \to
0 $ on the curve $X$. Since $S$ is orthogonal to $T$, we have $H^*(X, \ker
\pi)=0$. Thus $\mu(\ker \pi)=g-1$ by the Riemann-Roch Theorem. Since
$H^0(X,\ker \pi')$ is a subspace of $H^0(X, \ker \pi)$ this vector space is
zero. Hence, $\chi(\ker \pi') \leq 0$. We deduce that $\mu(\ker \pi') \leq g-1$. The inequality $\mu(\ker \pi')
\leq \mu(\ker \pi)$ translates as $\mu_\alpha(T') \leq \mu_\alpha(T)$ by
Lemma \ref{32}.
\end{proof}

\begin{proposition}\label{thmA2}
Let $T=(E_1 \rarpa \phi E_0)$ be a $\alpha$-semistable triple.
There exists a triple $S=(F_1 \rarpa \psi F_0)$ of type $\alpha^\bot$
which is orthogonal to $T$.
\end{proposition}
\begin{proof}
  Let $T=(E_1 \rarpa \phi E_0)$ be a $\alpha$-semistable triple.
  We deduce by Lemma \ref{lemmaET-T} that there exists a vector bundle $E_T$
  on the surface $X \times \pdopi$ 
  which is $H_\alpha$-semistable and in the following pull back diagram.
  \[ \xymatrix{ E_T \ar[r]\ar[d] & p^*E_1 \ar[d]^{p^*\phi} \\
p^*E_0 \otimes q^*(V \otimes \Ocal_\pdopi(-1)) \ar[r]^-\gamma & p^*E_0}\]
  Having in mind that $\gamma$ is surjective, we may rewrite this as a short
  exact sequence
  \[ \xymatrix{0 \ar[r] & E_T \ar[r] &  p^*E_1 \oplus p^*E_0 \otimes q^*(V
\otimes \Ocal_\pdopi(-1)) \ar[rr]^-{-p^*\phi + \gamma} && p^*E_0 \ar[r] & 0\,
.}\]
Now since $E_T$ is $H_\alpha$-semistable, by Theorem
\ref{BogStab} the restriction  of $E_T$ to a smooth curve $C$ in the linear
system $| m \cdot H_ \alpha |$ is also semistable for $m \gg 0$. Denoting the restriction of
$p$, and $q$ to $C$ by $\tilde p$, and $\tilde q$, respectively.
We remark that $\tilde p$ is an affin morphism.

We obtain the
short exact sequence
\[ \xymatrix{0 \ar[r] & E_T|_C \ar[r] &  \tilde p^*E_1 \oplus \tilde p^*E_0
\otimes \tilde q^*(V
\otimes \Ocal_\pdopi(-1)) \ar[rr]^-{-p^* \tilde \phi + \gamma} &&
\tilde p^*E_0 \ar[r] & 0\, .}\]

Now by Theorem \ref{Fal-thm} there exists a vector bundle $F$ on $C$ such that 
$H^*(X,  E_T|_C \otimes F)=0$. Tensoring the last exact sequence with $F$
gives after a push forward via $\tilde p$ the following vertical short exact sequence
on $X$.
\[ \xymatrix{0 \ar[d] \\
    \tilde p_*(E_T|_C \otimes F) \ar[d] \\
\tilde p_*(
\left( \tilde p^*E_1 \oplus \tilde p^*E_0
\otimes \tilde q^*(V
\otimes \Ocal_\pdopi(-1)) \right) \otimes F)
\ar[d]^-{-p^* \tilde \phi + \gamma} \ar@{=}[r]^-\sim &
E_1 \otimes \tilde p_*F  \oplus E_0 \otimes \tilde p_*( \tilde q^*(V \otimes
\Ocal_\pdopi(-1)) \otimes F) \ar[d]\\
\tilde p_*(\tilde p^*E_0 \otimes F) \ar[d] \ar@{=}[r]^-\sim
& E_0 \otimes \tilde p_*F \\ 0}\]
The horizontal isomorphisms follow from the projection formula.
Thus, defining the triple $S$ by $S=(F_1 \rarpa \psi F_0) := \tilde p_*( \tilde q^*(V \otimes
\Ocal_\pdopi(-1)) \otimes F \to F)$ gives the desired orthogonal triple.
The computation that $S$ is of type $\alpha^\bot$ is straightforward.
\end{proof}

\subsection{Effective bounds for orthogonal triples}
In the proof of Proposition \ref{thmA2} we did not used that for the curve
$C \in | m \cdot H_\alpha |$ the number $m$ can be given explicitly, as well
as the rank and the determinant of the vector bundle $F$ on $C$. Using these two
facts which follow from Theorems \ref{BogStab} and \ref{Pop-thm} we obtain:

\begin{proposition}\label{thmB2}
  Let $\alpha=\frac{a}{b}$ be a positive rational number with coprime $a,b \in
  \ndop$, and a line bundle $L$ on $X$ of degree one.
  Then for a holomorphic triple $T=(E_1 \to E_0)$ of rank $(r_1,r_0)$ and
  degree $(d_1,d_0)$ on $X$ we have the equivalence
  \begin{itemize}
    \item[(1)]  $T$  is $\alpha$-semistable.
    \item[(2)] there exists an orthogonal
      triple $S=(F_1 \rarpa \psi F_0)$ of type $\alpha^\bot$ with
      \[
        \begin{array}{rclcrcl} 
          \rk(F_0) & =&  m b(r_0+r_1)^2  & \qquad & \det(F_0)  & \cong &  L^{\otimes
        d} \\
        \rk(F_1) & = & 2m b(r_0+r_1)^2  & \qquad &  \det(F_1)  & \cong &
        L^{\otimes
        2d- 2(r_0+r_1)^2ma} \\
     \end{array} \]
     for $m = \max\{ 4(r_1d_0-r_0d_1) , \lceil \frac{2g}{a} \rceil \} $, and\\
     $d=m(r_0+r_1) \left( b(r_0+r_1)(g-1)
     -b(d_0+d_1) +2ar_0 \right) $.
  \end{itemize}
\end{proposition}

\begin{proof} We have only to show that (1) implies (2). The opposite
  implication follows from Theorem \ref{thmA}. So we assume that $T$ is
  $\alpha$-semistable.
  We consider the ample line $L_{a,b} = p^*L^{\otimes a} \otimes q^*\Ocal_\pdopi(b)$.
  By Lemma \ref{lemmaET-T}. The sheaf $E_T$ on $X \times \pdopi$ is semistable
  with respect to the polarization $L_{a,b}$. We computed the discriminant of
  $E_T$ to be $\Delta(E_T)=4(r_0d_1-r_1d_0)$. So our choice of $m$ guarantees
  that $L_{a,b}^{\otimes m}$ is globally generated, and the restriction of
  $E_T$ to the restriction of a smooth divisor $C$ in the associated linear
  system is also semistable by Theorem \ref{BogStab}. By Theorem \ref{Pop-thm}
  there exists a vector bundle $F$ on $C$ of rank $(r_0+r_1)^2$ and degree
  $\deg(F)=(r_0+r_1)^2(g_C-1)-(r_0+r_1)\deg(E_T)$. 
  Now we proceed like in the proof of Proposition \ref{thmA2}. That is we have
  $F_0=p_*(F)$ and $F_1= p_*(F \otimes q^*(V \otimes \Ocal_\pdopi(-1)))$.
  Indeed, since we can choose the determinant of $F$ arbitrary, we obtain
  $\det(F_0)= L^{\otimes d}$.
\end{proof}

\subsection{A homological view on orthogonal triples}
We follow here the notation of Weibel in \cite{Weibel}.
Considering two holomorphic triples $T=(E_1 \rarpa \phi E_0)$ and $S= (F_1
\rarpa \psi F_0)$ as chain complexes, then we see
that the {\em brutal truncation} $\sigma_{< 2}( T \otimes S)$ 
of the tensor product is the complex
\[ (E_1 \otimes F_0) \oplus (E_0 \otimes F_1) \rarpa \pi E_0 \otimes F_0 \]
we investigate in this Section.
Even though the brutal truncation is less common, we can describe the
orthogonality by
\[ H^*(X,H_1(\sigma_{< 2}( T \otimes S)))=0 \quad \mbox{, and} 
\quad H_0(\sigma_{< 2}( T \otimes S))=0 \, .\]
When $S$ is of type $\alpha^\bot$ the description becomes
\[ H^*(X,H_*(\sigma_{< 2}( T \otimes S)))=0 \, .\]

\section{The generalized theta line bundle for holomorphic triples}\label{S6}
\subsection{The theta line bundle on the stack} Let $T=(\Ecal_1 \rarpa \phi \Ecal_0)$ be
a triple on $X \times S$ with projections $\pr_X$ and $\pr_S$. For two fixed
classes $c_1=[F_1]$ and $c_0=[F_0]$ in the Grothendieck group $K(X)$ we define the
theta divisor $\Ocal_{S,T}(\Theta_{c_1,c_0})$ to be the determinant of cohomology
\[ 
  \scriptstyle{\Ocal_{S,T}(\Theta_{c_1,c_0}) = \det(R^*\pr_{S*}(\Ecal_0 \otimes \pr_X^*F_0 ))
  \otimes \det(R^*\pr_{S*}(\Ecal_1 \otimes \pr_X^*F_0 ))\inv
\otimes \det(R^*\pr_{S*}(\Ecal_0 \otimes \pr_X^*F_1 ))\inv \, .}
\]
Indeed, this line bundle depends only on the classes $c_i$ of the two bundles
$F_i$ in $K(X)$. Using the projection formula we see the following relation
between the theta line bundle for $T$ and $\pr^*L \otimes T= ( \pr_S^*L \otimes
\Ecal_1 \to \pr_S^*L \otimes \Ecal_0)$ for any line bundle $L$ on our base
scheme $S$:
\[ \Ocal_{S,\pr^*L \otimes T}(\Theta_{c_1,c_0}) =
    L^{\otimes \chi_{0,0}-\chi_{1,0}-\chi_{0,1}} \otimes
  \Ocal_{S,T}(\Theta_{c_1,c_0}) \]
  where $\chi_{i,j} = \chi_{X_0}(\Ecal_i \otimes \pr_X^* F_j)$ is the Euler
  characteristic of the vector bundle  $\Ecal_i \otimes \pr_X^* F_j$ on a
  fixed fiber $X_0$ of $\pr_S$. If we fix $\lambda \in k^*$ and denote the
  multiplication by $\lambda$ with $m_\lambda: T \to T$, then the induced
  isomorphism of $\Ocal_{S,T}(\Theta_{c_1,c_0}) \to
  \Ocal_{S,T}(\Theta_{c_1,c_0})$ is given by multiplication with
  $\lambda^{\otimes \chi_{0,0}-\chi_{1,0}-\chi_{0,1}}$.

  So by \cite[Th\'eor\`em  2.3]{DN} the line bundle
  $\Ocal(\Theta_{c_1,c_0})$ descends to the stable locus of the moduli space of
  triples, when the number $\chi_{0,0}-\chi_{1,0}-\chi_{0,1}$ is zero. We will
  assume from now on that $\chi_{0,0}-\chi_{1,0}-\chi_{0,1}=0$, and that we
  have an surjection $F_1 \rarpa \psi F_0$. To such a datum we will construct
  a theta divisor $\Theta_R$ where $R=(F_1 \rarpa \psi F_0)$.
  For any triple $T=(\Ecal_1 \rarpa \phi \Ecal_))$ on $X \times S$ we obtain a
  surjection
  \[ \xymatrix{ \pi: \Ecal_1 \otimes \pr_S^* F_0 \oplus \Ecal_0 \otimes \pr_S^* F_1 
  \ar[rrr]^-{-\phi \otimes \id_{F_0}+\id_{E_0}\otimes \psi } &&&  \Ecal_0 \otimes \pr_S^* F_0  }\]
  Now we define $\Theta_R$ to be the theta divisor associated to the vector
  bundle $\ker(\pi)$ as in the article \cite{DN} of Drezet and Narasimhan.
  This way we obtain a Cartier divisor  $\Theta_R$ on the moduli space of
  stable triples. The closed points of this divisor are give as
  \[ \Theta_R(k) =  \{ S= (E_1 \rarpa \phi E_0) |\, S \mbox{ is not orthogonal
  to } R \, \} .\]

  \subsection{Base point freeness}
  Proposition \ref{thmB2} yields a base point free result for these
  generalized theta divisors. Let $M=M^\alpha_{(r_1,r_0,d_1,d_0)}$ be the moduli
  space of $\alpha$-semistable triples $T=(E_1 \rarpa \phi E_0)$ with
  $\rk(E_i)=r_i$, and $\deg(E_i)=d_i$, then the theta divisor $\Theta_R$ is
  base point free for $R=(F_1 \rarpa \psi F_0)$ for sheaves $F_i$ as in
  Proposition \ref{thmB2}.

\subsection{A criterion for ampleness}
We fix the numerical date $c=(r_0,r_1,d_0,d_1,\alpha)$ and obtain a map
from semistable $c$-triples on $X$ to semistable $\tilde c$-triples on the
curve $C \subset X \times \pdopi$ in the linear system $|m \cdot H_\alpha |$.
Since the theta divisor on the moduli space of S-equivalence classes of rank
$r_0+r_1$ bundles on $C$ is ample, we can hope that this also holds for its
pull back to the moduli space of S-equivalence classes of $c$-triples.
However, it turns out that we need an additional condition to ensure this.

\begin{proposition}\label{prop-ample}
  Let $R=(F_1 \rarpa \psi F_0)$ be a triple satisfying:\\
  \begin{tabular}{lll}
    $\qquad $ &(1) &  $\psi$ is surjective,\\
              &(2) & $\mu(F_0)-\mu(F_1)=\alpha$,\\
              &(3) & $\rk(F_1) =2\rk(F_0)$, and\\
              &(4) & $\mu(F_0)+\mu_\alpha = g-1$.\\
  \end{tabular}\\
  If for all semistable triples $T=(E_1 \to E_0)$ we have $\Hom(E_0,E_1)=0$,
  then the divisor $\Theta_R$ is ample.
\end{proposition}
\begin{proof}
  The construction of the theta divisor uses the morphism
  \[ \rho : \{ \mbox{semistable triples of type } c \mbox{ on } X \} \to \{ 
  \mbox{ semistable vector bundles on } C \} \]
  where $C$ is a curve in the linear system $|m \cdot H_\alpha| $.
  Since the theta divisor is ample on the moduli space of semistable bundles
  on $C$, it suffices to show that $\rho$ is infinitesimal injective.
  The morphism $\rho$ can be decomposed into three steps:

  \begin{tabular}{cp{12cm}}
    (1) & We assign the triple $T=(E_1 \to E_0) $ on $X$  the short exact
    sequence  $(0 \to E_0
    \boxtimes \Ocal_\pdop(-2) \to E_T \to p^*E_1 \to 0)$ on $X \times
    \pdopi$.\\
    (2) & We go from $(0 \to E_0
        \boxtimes \Ocal_\pdop(-2) \to E_T \to p^*E_1 \to 0)$ to the vector
        bundle $E_T$ on $X \times \pdopi$.\\
    (3) & We go to the restriction of $E_T$ on the curve $C$. 
  \end{tabular}

  The first step is an equivalence as we have seen in part \ref{22}. To see
  that the third step is infinitesimal injective, we remark first that we can
  choose a curve $C$ in a linear system $|m \cdot H_\alpha|$ for $m \gg 0$.
  However, if $H^1(X \times \pdopi, \End(E_T)(-m) )=0$, then the tangent map
  $\Ext^1(E_T,E_T) \to \Ext^1(E_T|_C,E_T|_C)$ is injective. Since the
  semistable triples form a bounded family we can choose an integer $m$ which
  works for all triples of the fixed type $c$.

  Now study infinitesimal injectivity of step (2). Assume that we have a
  deformation of the short exact sequence $(0 \to E_0
      \boxtimes \Ocal_\pdop(-2) \to E_T \to p^*E_1 \to 0)$ over $X \times
      \pdopi \times \Spec(k[\eps])$ which is constant when we consider only
      the deformation of $E_T$. This defines a tangent vector in the 
      Quot scheme in the point $[ E_T \to p^*E_1 ]$.
      The tangent space $T_{[ E_T \to p^*E_1 ]}$ is given by
      \[ T_{[ E_T \to p^*E_1 ]} = \Hom(E_0 \boxtimes \Ocal_\pdopi(-2),p^*E_1)
      = H^0(\pdopi,\Ocal_\pdopi(2)) \otimes \Hom(E_0,E_1) \, .\]
      Thus, the condition $ \Hom(E_0,E_1)=0$ forces the infinitesimal
      injectivity of $\rho$ which gives the assertion.
\end{proof}

\section{An example}\label{S7}
In this section $X$ denotes a curve of genus $g=3$. We study the coarse moduli space
$M=M^\frac{3}{2}_X(2,1,1,2)$
of semistable pairs $T=(E_1 \rarpa \phi E_0)$ where
\[ \rk(E_1)=2, \quad \deg(E_1)=1 \quad \rk(E_0)=1, \quad \deg(E_0)=2 \]
for $\alpha= \frac{3}{2}$.
Checking the few possible numerical types of sub triples we find that
\begin{proposition}\label{p71}
  For a holomorphic triple $T=(E_1 \rarpa \phi E_0)$ we have the following
  equivalence:
\[ T \mbox{ is (semi)stable} \iff \begin{array}{lrcll}
    (1) & \phi &  \ne& 0 & \mbox{, and} \\
    (2) & \deg(\ker(\phi))&(\leq)&0  & \mbox{, and} \\
    (3) & \mu_{\max }(E_1)&( \leq )&1 & .
  \end{array}
\]
\end{proposition}
We conclude two corollaries:
\begin{corollary}\label{c72}
  $T$ is semistable $\implies \Hom(E_0,E_1) =0 $.
\end{corollary}
\begin{corollary}\label{c73}
$T$ is stable $\iff E_1$ is stable, and $\phi$ is surjective. Thus, any stable
triple $T$ is of type $\alpha^\bot$.
\end{corollary}
\begin{proof}
  The surjectivity and the stability of $E_1$ is a consequence of Proposition
  \ref{p71}. The equalities $2\rk(E_0)=\rk(E_1)$ and
  $\mu(E_0)-\mu(E_1)=\alpha$ are obvious.
\end{proof}
For a fixed triple
$ S=(F_1 \rarpa \psi F_0) \in M(k)$ we defined the corresponding
$\Theta$-divisor by
\[ \Theta_S:= \left\{ (E_1 \rarpa \phi E_0) \in M(k)  \left| H^* \left( \ker
    \left(
      \xymatrix{ E_1 \otimes F_0 \oplus E_0 \otimes F_1
    \ar[rr]^-{ \phi \otimes \id_{F_0} - \id_{E_0}\otimes \psi} && E_0 \otimes
  F_0 } \right. \right) \right) \ne
  0 \right\}
\]
In this special case we have the next
\begin{proposition}
  For any $S \in M(k)$ the generalized theta divisor $\Ocal_M(\Theta_S)$ is ample.
  The divisor $\Theta_S$ always contains the point $S$.
\end{proposition}
\begin{proof}
  We showed in Proposition \ref{prop-ample} that Corollary \ref{c72} implies
  the ampleness of $\Theta_S$.

  The observation that the point $S=(F_1 \rarpa
  \psi F_0)$ is always
  contained in $\Theta_S$ follows from $H^0(G) \ne 0$ where $G$ is the kernel
  of the morphism
  \[ \xymatrix{F_1 \otimes F_0 \oplus F_0 \otimes F_1 \ar[rrr]^-{\psi \otimes
    \id_{F_0} - \id_{F_0} \otimes \psi} &&& F_0 \otimes F_0 \, . } \]
        We have $\mu(F_1 \otimes F_0) = \frac{5}{2}>g-1$. Therefore we have
        $H^0(F_1 \otimes F_0)$ has a positive dimension. Let $e \ne 0 $ be any
        element in $H^0(F_0 \otimes F_1)$. Let $\sigma: F_1 \otimes F_0
        \rarpa \sim F_0 \otimes F_1$ be the interchanging isomorphism. The
        element $(e, H^0(\sigma)(e))$ is a non trivial global section of $G$.

      \end{proof}

        Proposition \ref{thmB2} implies that $216 \cdot \Theta_S$
        is globally generated. However, this is far from a good bound.

\begin{lemma}\label{l-ss-or}
  If the tripe $T=(E_1 \to E_0)$ is semistable but not stable, then there
  exists a triple $S=(F_1 \to F_0) \in M(k) $ which is orthogonal to $T$.
\end{lemma}
\begin{proof}
  If $T$ is not stable it contains a proper sub triple $T' $ with
  $\mu_\alpha(T')=0$. There are two such possibilities:
  Either the triple $T'=T_0=(L_0 \to 0) $ with $\deg(L_0)=0$, or we have $T'= T_1
  =(L_2(-Q) \to L_2)$ with $\deg(L_2)=2$ and $Q \in X(k)$.
  Since $T$ sits in a short exact sequence 
  \[ 0 \to T_i \to T \to T_{1-i} \to 0 \]
  with $i=0$ or $i=1$, it is enough to find a triple $S$ which is orthogonal to
  $T_0$ and $T_1$.
  Indeed, we can choose $S=S_0 \oplus S_1$ as follows:\\
  $S_0=(M_0 \to 0)$ with $M_0$ a degree zero line bundle. For any $M_0$ we have that
  $S_0$ is orthogonal to $T_0=(L_0 \to 0)$. $S_0$ is orthogonal to $T_1=(L_2(-P) \to L_2)$ iff
  $H^*(L_2 \otimes M_0) = 0$. Thus, for a general $M_0$ the triple $S_0$ is
  orthogonal to $T_0$ and $T_1$.\\
  $S_1=(M_2(-Q) \to M_2)$ with $M_2$ of degree two. We have that $(L_0 \to 0)$
  is orthogonal to $S_1$ iff $H^*(L_0 \otimes M_2)=0$, again an open condition
  for $M_2$. For the orthogonality of $S_1$ to $T_1$ we need two things: First
  of all the morphism $\pi: L_2(-P) \otimes M_2 \oplus L_2 \otimes M_2(-Q)
  \to L_2 \otimes M_2$ must be surjective. This is the case whenever $P \ne
  Q$. Under this condition, the kernel of $\pi$ is the line bundle $L_2 \otimes
  M_2(-P-Q)$. Again for a general $M_2$ we have $H^*(\ker(\pi))=0$.
\end{proof}

\begin{lemma}\label{l-st-or}
  If the tripe $T=(E_1 \rarpa \phi E_0)$ is stable, then there
  exists a triple $S=(F_1 \to F_0) \in M(k) $ which is orthogonal to $T$.
\end{lemma}
\begin{proof}
  By Corollary \ref{c73} we know that $E_1$ is stable and $\phi$ is surjective.
  Our strategy will be similar to the above proof. We show that $T$ is
  orthogonal to a direct sum $S=S_0 \oplus S_1$. As before we set $S_0=(M_0
  \to 0)$ and $S_1=(M_2(-Q) \to M_2)$.\\
  The orthogonality of $T$ to $S_0$ is equivalent to the vanishing of $H^*(M_0
  \otimes E_0)$ which is true for a general $M_0$.\\
  Next we look for a triple $S_1$ which is orthogonal to $T$.
  We choose $M_2$ such that:\\
  (1) $h^0(M_2 \otimes E_0) = \chi(M_2 \otimes E_0)=2$,\\
  (2) $h^0(M_2 \otimes \ker(\phi))=0$, and\\
  (3) $h^0(M_2 \otimes E_1) = \chi(M_2 \otimes E_1)=1$.\\
  The first two conditions are obviously satisfied for a general $M_2$, and
  are open conditions on $\Pic^2(X)$. For condition (3) we use any surjection
  $E \rarpa \pi  k(P)$. The kernel of $\pi$ is semistable of rank two and degree
  zero. By Raynaud's result in \cite{Ray} for a general $M_2$ we have $H^*(M_2 \otimes
  \ker(\pi))=0$. This implies (3).
  For a line bundle $M_2$ satisfying (1)--(3) we have that 
  the morphism $\sigma: H^0(M_2 \otimes E_1) \to  H^0(M_2 \otimes E_0)$ is
  injective with image spanned by a global section $s \in  H^0(M_2 \otimes
  E_0)$. Now we take a point $Q \in X(k)$ such that $s$ is not in the image
  of $H^0(M_2(-Q) \otimes E_0)  \to H^0((M_2 \otimes E_0)$.
  It follows that the morphism 
  \[ H^0(M_2 \otimes E_1) \oplus H^0(M_2(-Q) \otimes E_0)  \to H^0((M_2
  \otimes E_0) \]
  is an isomorphism. Thus, $S_1=(M(-Q) \to M)$ is orthogonal to $T$.
\end{proof}

\end{document}